\documentclass[preprint,12pt,3p]{elsarticle}
\usepackage{color}
\usepackage{amssymb}
\usepackage{amsmath}

 \usepackage{amsthm}

\usepackage{amsmath,amsfonts}
\DeclareMathOperator{\spn}{span}
 \newtheorem{thm}{Theorem}[section]
 
 \newtheorem{lem}[thm]{Lemma}
 \newtheorem{prop}[thm]{Proposition}
 \numberwithin{equation}{section}
\usepackage{color}

\usepackage{xcolor}
\usepackage{xwatermark}
\usepackage{soul}

\begin{document}

\begin{frontmatter}

\title{Sharp estimates for the covering numbers of the Weierstrass fractal kernel}

\author[label1]{Douglas Azevedo%\corref{cor2}
%\fnref{label1} 
}
\ead{douglasa@utfpr.edu.br}

\address[label1]{DAMAT-UTFPR. Av. Alberto Carazzai, 1640, 863000-000, Cornelio Procopio - Brazil}

\author[label2]{Karina Gonzalez\corref{cor1}
%\fnref{label2}
}
\cortext[cor1]{I am corresponding author}
\ead{karina.navarro@correounivalle.edu.co}

\address[label2]{ICMC - Universidade de S\~{a}o Paulo, Brazil}

\author[label3]{Thaís Jordão%\corref{cor2}
%\fnref{label2}
}
\ead{tjordao@icmc.usp.br}
\address[label3]{ICMC - Universidade de S\~{a}o Paulo.\\
Av. Trabalhador saocarlense, 400, São Carlos, S\~{a}o Paulo, Brazil, 13566-590}

\begin{abstract}
In this paper, we present sharp estimates for
the covering numbers of the embedding of the reproducing kernel Hilbert space (RKHS) associated with the Weierstrass fractal kernel into the space of continuous functions. The method we apply is based on the characterization of the infinite-dimensional RKHS generated by the Weierstrass fractal kernel and it requires estimates for the norm operator of orthogonal projections on the RKHS.
\end{abstract}


\end{frontmatter}

\section{Introduction}
\label{intro}

In 1872, K. Weierstrass presented a special class of continuous and nowhere differentiable functions (CNDFs) given by Fourier/trigonometric series expansions.\ The Weierstrass function is defined as follows
\begin{equation}\label{1.1} 
w_{a,b} (x)= \sum_{n=0}^{\infty} a^n \cos(b^n \pi  \,x), \quad x\in\mathbb{R},
\end{equation}
where $0<a<1$ and $b\in \mathbb{R}$. It is clear that the Weierstrass function defines a continuous bounded function.\ Weierstrass proved that $w_{a,b}$ is nowhere differentiable provided that $ab\geq 1+3\pi/2$,  with $b$ an odd integer.\ G.H. Hardy relaxed this condition in \cite{hardy} and proved that for $0<a<1$ and $ab\geq 1$, the Weierstrass function remains nowhere differentiable.\ This function plays an important role in fractional Brownian motion  and it inspires the Weierstrass-Mandelbrot curve, a fractal curve widely explored in fractal geometry (\cite{BL,PT}). In this paper,  the Weierstrass function is considered to design the Weierstrass fractal kernel.\ 

Consider $I:=[-1,1]$ and $C(I)$ the space of continuous real-valued functions on $I$ endowed with the supremum norm $\|\, \cdot \,\|_{\infty}$.\  Throughout the paper, for $b\in\mathbb{N}$ we will consider the orthonormal system $$
\{\cos(b^{n}\pi \, \cdot \, ),\sin(b^{n}\pi \, \cdot \, ): n\in \mathbb{N}\}
$$ in $L^2(I)$, where $L^2 (I):= L^2 (I,dx)$ is the usual Hilbert space of square integrable functions with the usual inner product $$
\langle f,g  \rangle_2 = \int_{-1}^{1} f(x)g(x)\, dx, \quad f,g \in L^2 (I). 
$$

For $0<a<1$ and $b$  a natural number such that $ab\geq 1$, the \emph{Weierstrass fractal kernel}  $W: I\times I\longrightarrow \mathbb{R}$ is given  by
\begin{equation}\label{WK}
W(x,y):=w_{a,b}(x-y), \quad x,y\in I,
\end{equation}
where $w_{a,b}$ is the Weierstrass function in equation \eqref{1.1}.\ This kernel is continuous, nowhere differentiable, and positive definite.\ The theory of RKHSs (see \cite{aron}) ensures that there exists a unique Hilbert space (RKHS) $\mathcal{H}_W :=(\mathcal{H}_W (I), \langle \cdot, \cdot\rangle_{W})$ such that the Weierstrass fractal kernel is the associated reproducing kernel. 

We present upper and lower (sharp) estimates for the covering numbers of the embedding $I_W : \mathcal{H}_W \rightarrow C(I)$. The covering numbers of an operator   are defined in terms of the covering numbers of subsets of metric spaces. If $A$ is a subset of a metric space $X:=(X,d)$ and $\epsilon>0$, then the covering numbers $\mathcal{C} (\epsilon, A)$ are the minimal number of balls in $X$ of radius $\epsilon$ covering $A$.\ It easy to see that if $A$ is a compact subset of a metric space, then $\mathcal{C}(\epsilon, A) <\infty$.\ If we consider  $(X, \|\cdot\|_X), (Y,\|\cdot\|_Y)$ Banach spaces and $B_X$ the unit ball in $X$, then the \emph{covering numbers} of an operator $T: X \rightarrow Y$ are given by
\begin{equation}\label{covnumop}
\mathcal{C} (\epsilon, T):= \mathcal{C}(\epsilon,T(B_X)), \quad \epsilon>0.
\end{equation}

The generality of the concept of metric entropy (\cite{KT,Lorentz1}) includes covering numbers, packing numbers, entropy numbers and it has several applications in many branches of Mathematics (see  \cite{Edmunds,haussler,Lorentz2,Konig,Li,SC}).\ In information theory and statistics, including machine learning, it plays a central role.\ We suggest \cite{zhou1} and references quoted there for more information on metric entropy and machine learning methods.

Our main result asserts that \emph{Kolmogorov’s $\epsilon$-entropy} (\cite{KT}) $\ln (\mathcal{C} (\epsilon, I_W))$ is asymptotically  
equivalent to $\left[\ln(1/\epsilon)\right]^2$. In the paper, for functions $f,g: (0,\infty) \rightarrow \mathbb{R}$, the notation $f(\epsilon) \asymp g(\epsilon)$ stands for
\[
0< \lim_{\epsilon\rightarrow 0} \inf \frac{f(\epsilon)}{g(\epsilon)} \leq \lim_{\epsilon\rightarrow 0} \sup \frac{f(\epsilon)}{g(\epsilon)} < \infty,\]
called weak equivalence (\cite{KT}).\ We will also use the strong equivalence notation $f(\epsilon) \approx g(\epsilon)$ and it means that
%are considered 
\[
\lim_{\epsilon\rightarrow 0} \frac{f(\epsilon)}{g(\epsilon)}=1.
\]

The main result of this paper reads as follows.

\begin{thm}\label{coveringesatimates} Let $W$ be as in \eqref{WK}.\ The covering numbers of the embedding $I_{W}:\mathcal{H}_W \rightarrow C(I)$ satisfy
\[ \frac{2}{{\ln\left(1/a\right)}}  \leq \liminf_{\epsilon\to 0}\frac{\ln(\mathcal{C}(\epsilon, I_W ))}{[\ln \left(1/\epsilon \right)]^2} \leq \limsup_{\epsilon\to 0}\frac{\ln(\mathcal{C}(\epsilon, I_W ))}{[\ln \left(1/\epsilon \right)]^2} \leq \frac{4}{{\ln\left(1/a\right)}}.\]
\end{thm}

In particular, Theorem \ref{coveringesatimates} provides a weak equivalence for the covering numbers of a new example of non-trivial kernel, given by
\[\ln(\mathcal{C}(\epsilon, I_W ))\asymp \left[ \ln \left(1/\epsilon \right) \right]^2,  \quad \mbox{as}\quad \epsilon \rightarrow 0.\] 
The technique we employ to prove Theorem \ref{coveringesatimates} is based on the estimate of the operator norm of $I_{W}$ and some other related finite rank operators.\ This
approach follows the path employed in \cite{Azevedo}, where the authors investigated estimates for the covering numbers of the unit ball of the RKHS associated with isotropic kernels $W$ on compact two-point homogeneous spaces.\ On the other hand, in \cite{Kuhn} a similar technique was employed in order to obtain the asymptotic behavior of the covering numbers of the embedding of the RKHS associated to the Gaussian kernel on subsets of $\mathbb{R}^d$ in the space of continuous functions.\ There is a large literature on this subject and a constant interest on (optimal/sharp) estimates for entropy numbers, we refer \cite{dryanov,KT,Kuhn,Kuhn4,Kuhn5,Kuhn6,Kuhn7, pozharska1,pozharska2} and references therein.

The paper is organized as follows.\ In Section \ref{WKer}, we present some general facts about RKHSs and the characterization of $\mathcal{H}_{W}$. In particular, we present an orthonormal basis for $\mathcal{H}_{W}$. In Section 3, we obtain the upper and the lower bounds for the covering numbers and Theorem \ref{coveringesatimates} is proved.  

\section{The Weierstrass fractal kernel and its RKHS}
\label{WKer}\label{WRKHS}

In this section we present some necessary information and properties about the Weierstrass fractal kernel $W$ introduced in \eqref{WK}.\ We discuss some basic facts about positive definite kernels and their associated RKHSs. The main references for this section are \cite{aron,CZ}.

Let $X$ be a non-empty set.\ A symmetric function $K: X \times X \rightarrow \mathbb{R}$ is called a \emph{positive definite kernel} if it satisfies 
$$\sum_{i=1}^{n}\sum_{j=1}^{n} c_i c_j K(x_i,x_j ) \geq 0$$
for all $n\geq 2, \{ x_1 , x_2 ,...,x_n \}\subset X$ and $\{c_1 ,c_2 ,...,c_n \}\subset \mathbb{R}$.\ 

If $K: X \times X\to \mathbb{R}$ is a positive definite kernel, then there is a (unique)  RKHS $\mathcal{H}_K$ induced by $K$. This Hilbert space is given by the closure of the $\spn\{ K(y, \,\cdot \,) : y \in X\}$ with respect to an inner product $\langle{\,\cdot \,},{\, \cdot \,}\rangle_{K}$, satisfying the following properties 
\begin{itemize}
    \item[\textbf{R1.}] The function $x\mapsto K(x,\cdot)$ belongs to $\mathcal{ H}_K$, for all $x\in X$; 
    \item[\textbf{R2.}] (Reproduction property): $f(x)=\langle  f,K(x,\cdot) \rangle_{K}$, for all $f\in \mathcal{ H}_K$ and $x\in X$. 
\end{itemize}

It is straightforward to see
that the Weierstrass fractal kernel is positive definite.\ For $0<a<1$, the series in formula (\ref{WK}) converges uniformly on $\mathbb{R}\times\mathbb{R}$ and it implies that $W$ is a continuous kernel on $I$.\  Finally, $W$ is also nowhere differentiable on $I\times I$, since the partial derivatives of $W$ do not exist as a consequence of the nowhere differentiability of $w_{a,b}$ (\cite{JP,johsen}).\ These properties are given in the following result. 

\begin{lem}\label{LEMAPDW}
The function $W: I\times I\rightarrow \mathbb{R}$ is a continuous, positive definite and nowhere differentiable kernel.
\end{lem}

We proceed by characterizing the RKHS $\mathcal{ H}_W$.\ We will write $\ell^2$ for the usual Hilbert space of the square summable sequences of real numbers, endowed with the standard inner product.\ The theorem asserts that the system 
\begin{equation} \label{baseRKHS}
\{a^{n/2}\cos(b^{n}\pi \, \cdot \,),\, a^{n/2}\sin(b^{n}\pi \, \cdot \,): n\in \mathbb{N}\}, 
\end{equation}
is an orthonormal basis of the space $\mathcal{H}_{W}$ (see \cite[Lemma 2.6]{SS}).

\begin{thm}\label{teo1} The RKHS $\mathcal{H}_W$ induced by the Weierstrass fractal kernel is given by
%\begin{equation*}
% \begin{split}
%\mathcal{H}_W :=\left\{ f \sim \sum_{n=0}^{\infty} %a^{n/2} \left[c_n \cos(b^n \pi \,\,\cdot\,\,)+  d_n \sin(b^n \pi \,\,\cdot\,\,)\right] : \{(c_n ,d_n ) \}\in (\ell \times \ell)^2 \right\},
%\end{split}
%\end{equation*}
\begin{equation*}
\left\{ f: f(x) = \sum_{n=0}^{\infty} a^{n/2} \left[c_n \cos(b^n \pi x)+ d_n \sin(b^n \pi x)\right], \,\, x\in I \,\, \mbox{and} \,\, \{c_n\},\{d_n\}\in \ell^2 \right\}
\end{equation*}
endowed with the inner product
$$
\langle  f,g \rangle_{W} :=\sum_{n=0}^{\infty}\left(c_n e_n +d_n f_n\right), \quad f,g\in \mathcal{H}_W,
$$
where
%$$f \sim \sum_{n=0}^{\infty} a^{n/2} \left[c_n \cos(b^n \pi \,\cdot \,\,)+ d_n \sin(b^n \pi  \,\,\cdot\, \,)\right]$$ and $$g\sim \sum_{n=0}^{\infty} a^{n/2} \left[e_n \cos(b^n \pi \,\,\cdot \,\,)+ f_n \sin(b^n \pi \,\,\cdot \,\,)\right].
%$$ 
$$f(\, \cdot\,)= \sum_{n=0}^{\infty} a^{n/2} \left[c_n \cos(b^n \pi \, \cdot\,)+ d_n \sin(b^n \pi \, \cdot\,)\right]$$ and 
$$g(\, \cdot\,)= \sum_{n=0}^{\infty} a^{n/2} \left[e_n \cos(b^n \pi \, \cdot\,)+ f_n \sin(b^n \pi \, \cdot\,)\right].
$$ 
\end{thm}

\begin{proof}
Let $\mathcal{H}$ be the set of functions given in the theorem. It is not difficult to see that $(\mathcal{H}, \langle \cdot, \cdot \rangle_W)$ is a Hilbert space.\ Due the unicity of the RKHS induced by $W$ we just need to show that properties R1 and R2 above are satisfied. 

For $n\geq 0$ an integer, we write 
$$
c_n(x):={a^{n/2} \cos(b^n \pi x)}\quad \mbox{and} \quad d_n(x):={a^{n/2} \sin(b^n \pi x)},\quad x\in I.
$$
Consider $x\in I$. Since the sequences $\{c_n(x)\}$ and $\{d_n(x)\}$ belong to $\ell^2$, we have that
$$
W(x, \, \cdot \,) = \sum_{n=0}^\infty a^{n/2} \left[ c_n(x)  \cos(b^n \pi \,\, \cdot \,\,) + d_n(x) \sin(b^n \pi \,\, \cdot \,\,)\right]
$$
belongs to $\mathcal{H}$ and, in consequence, $\spn\{W(x,\, \cdot \,),\,x\in I\}$ is a dense subspace of $\mathcal{H}$ and property R1 is satisfied. 

In order to show R2, for every $f\in \mathcal{H}$, we write it in its series representation 
$$f(x) = \sum_{n=0}^{\infty} a^{n/2} \left[c_n \cos(b^n\pi x) +  d_n \sin(b^n\pi x)\right],\quad x\in I$$
and it is easy to see that
$$
\langle f, W(x, \cdot ) \rangle_{W} %= \sum_{n=0}^{\infty} a^{n/2} c_n \cos(b^n \pi x)+ a^{n/2} d_n \sin(b^n \pi x) 
=f(x), \quad x\in I.
$$
Therefore, $\mathcal{H}=\mathcal{H}_W$ and the proof follows.
\end{proof}

It is important to note that Theorem \ref{teo1} gives us a Fourier-like series characterization for the elements of $\mathcal{H}_{W}$.\ As a consequence of \cite[p. 22]{johsen}, functions in $\mathcal{H}_{W}$ of the form
\[
f(\,\cdot\,)= \sum_{n=0}^{\infty} a^{n/2} \left[c_n\cos(b^n \pi \,\cdot\,)+ d_n \sin(b^n \pi \,\cdot\,)\right],
\]
are nowhere differentiable, if both $\lim_{n\rightarrow\infty}c_n (a^{1/2} b)^n$ and $\lim_{n\rightarrow\infty} d_n (a^{1/2} b)^n$ are non zero. In particular, it can be shown that the collection of functions in $\mathcal{H}_{W}$ satisfying these properties is a dense proper subset of $\mathcal{H}_{W}$.

\section{Estimates for the covering numbers}

We proceed by splitting this section in three subsections in order to prove the estimates for the covering numbers $\mathcal{C}(\epsilon, I_W )$.\ The upper bound is obtained in Subsection 3.1, and the lower bound in Subsection 3.2. Finally, in Subsection 3.3  we present the proof of Theorem \ref{coveringesatimates} as a direct consequence of the previous estimates and final comments.\ 

The following lemma presents important properties of the covering numbers that we will use in this section (for more details see \cite[Lemma 1]{Kuhn} and \cite[Section 2.d]{Konig}).\ 
 
 \begin{lem} \label{LemCN}
Consider $S, T: X\rightarrow Y$ and $R: Z \rightarrow X$ operators on real Banach spaces.\ For any $\epsilon, \delta > 0$ the following properties hold: 
 
\begin{itemize}
 \item[i)] $\mathcal {C} (\epsilon + \delta, T+S) \leq \mathcal{C} (\epsilon, T) \hspace{0.1cm} \mathcal{C}(\delta, S)$,

 \item[ii)] if $rank(T) <\infty$, then $\mathcal{C}(\epsilon, T)\leq \left( 1+2\| T\|/\epsilon \right)^ {rank (T)}$, 
 
 \item[iii)] if $\| T\| \leq \epsilon $, then $\mathcal{C}(\epsilon, T)=1$,
 
\item[iv)] if $\epsilon <\delta$, then $\mathcal{C}(\delta,T)<\mathcal{C}(\epsilon, T)$ and

\item[v)]  $\mathcal{C}(\epsilon \delta,TR)\leq \mathcal{C}(\epsilon, T)\hspace{0.1cm}\mathcal{C}(\delta, R)$.
\end{itemize}
 \end{lem}

\subsection{Upper bound}

From the previous lemma, we see that the operator norm can be used to achieve upper bounds for the covering numbers of the operator, then $\|I_W\|$ plays a crucial role in our approach.

\begin{prop}\label{NormEmb}
 The embedding $I_W : \mathcal{H}_W \longrightarrow C(I)$ satisfies 
 $$
  \| I_W \|^2  = \frac{1}{1-a}.
  $$
\end{prop}
 
\begin{proof}
The reproducing property R2 of $W$ in $\mathcal{H}_W$ and an application of the Cauchy-Schwarz inequality leads us to
\begin{align*}
\| I_W  \|^2 \leq  \sup _{x\in I} \| W(x,\cdot)\|_W  ^2
   =  \frac{1}{1-a}.
\end{align*}
If we consider $f\in \mathcal{H}_W$, given by
$$
f(\,\cdot\,)= \sum_{n=0}^\infty a^{n/2} \left[c_n \cos(b^n \pi \,\cdot\,) + d_n \sin(b^n \pi \,\cdot\,)\right],
$$
then $g(x):= f(x)/\|f\| _W$, $x\in I$, is such that $g\in\mathcal{H}_W$ and $\|g\|_W=1$.\ For $f$ as above and $c_n =a^{n/2}$, $n\in\mathbb{N}$, we have that $g(0)=1/(1-a)$ and the proof follows. 
\end{proof}

Similarly to the proof above we can determine the operator norm of projections on orthogonal subspaces of $\mathcal{H}_W$.\ For $N>1$ a natural number, we write $P_{\mathcal{U}}^N$ and $P_{\mathcal{V}}^{N}$ for the projections onto 

\begin{align*}
\mathcal{U}:=\textrm{span}&\, \left\{ a^ {n/2} \cos (b^n \pi \,\,\cdot \,\,) , a^ {n/2} \sin(b^n \pi\,\,\cdot \,\,) : n= 0, \ldots ,N-1\right\}
\end{align*}
and $\mathcal{V}$, the closure of the $\mbox{span} \left\{ a^ {n/2} cos (b^n \pi\,\,\cdot \,\,) , a^ {n/2}  \sin(b^n \pi\,\,\cdot \,\,) :  n= N, \ldots \right\}$, 
respectively. 

\begin{lem}\label{projnorm} If $P_{\mathcal{U}}^N$ and $P_{\mathcal{V}}^{N}$ are the projections defined above, then 
\[\| I_W P_{\mathcal{U}}^N \|^2 = \frac{1- a^N}{1-a} \quad \mbox{and} \quad \| I_W P_{\mathcal{V}}^{N} \|^2 = \frac{a^N}{1-a}.\]
\end{lem}

From now on, we will employ the notation 
$$
\mu_N := \left( \frac{a^N}{1-a}\right)^{1/2}, \quad N \geq 1.
$$
%For the next result we observe that $\ln (\mathcal{C} (\epsilon, I_W)) \asymp  \ln (\mathcal{C} (2\epsilon, I_W))$, as $\epsilon\rightarrow 0$. This is a simple application of Lemma \ref{LemCN}, items $iv)$ and $v)$. 
The upper estimate for the covering numbers is given as follows.

\begin{thm}\label{tma4.4}
For any $\epsilon > 0$ sufficiently small, $$\ln (\mathcal{C} (\epsilon, I_W))\leq  4\frac{[\ln \left(1/\epsilon \right)]^2}{\ln(1/a)}.$$
\end{thm}

\begin{proof}  First, note that  
\begin{equation*} \label{4.5}
    \frac{\ln (  (1-a) \mu_N^2 )}{\ln a} = N, \quad N\geq 1.
\end{equation*} 
Consider $\epsilon >0$ sufficiently small and let $N=N(\epsilon)$ be a natural number such that 
\begin{equation*}\label{MuNEpsilon}
    \mu_N \leq \epsilon/2< \mu_{N-1}.
\end{equation*}
It is not hard to see that 
\begin{equation}\label{Neps}
N(\epsilon) \approx \frac{\ln({\epsilon}^2 (1-a)/4)}{\ln a}\approx \frac{2\ln(1/\epsilon )}{\ln (1/a)}.
\end{equation}

From Lemma \ref{projnorm} we have $\| I_W P_{\mathcal{V}}^{N} \| = \mu_N$. By an application of Lemma \ref{LemCN}, item $iii)$, we obtain that
$$
\mathcal{C}(\mu_N ,I_W P_{\mathcal{V}}^{N} )=1.
$$
Since ${rank}(I_W P_{\mathcal{U}}^N )=2N$,  Lemma \ref{LemCN}, item $ii)$, implies that 
\begin{eqnarray*} 
\mathcal{C} (\epsilon/2, I_W P_{\mathcal{U}}^N ) %& \leq & \left(1+ \frac{2\| I_W P_{\mathcal{U}}^N \| }{ \epsilon} \right)^ {\text{rank}(I_W P_{\mathcal{U}}^N )}\\ 
& \leq & \left( 1+\frac{{4}}{ {\epsilon }} \left( \frac{{1-a}^N}  {1-a} \right)^{1/2}\right)^{2N}.
\end{eqnarray*}
%By \eqref{MuNEpsilon}, we obtain that
%From the trivial estimate
%\begin{eqnarray*} 
% \left( 1+\frac{{2}}{ {\epsilon }} \left( \frac{{1-a}^N}  {1-a} \right)^{1/2}\right)^N & = & \left[ 1+ \frac{2}{\epsilon} \left( \frac{1}{1-a} - {{{ \mu_N ^2 }}}\right)^{1/2} \right]^ N\\ 
                            % & < & \left[ 1+ \frac{2}{\epsilon} \left( \frac{1}{1-a} \right)^{1/2} \right]^ N. 
%\end{eqnarray*}
Combining these estimates and Lemma \ref{LemCN}, item $i)$ we have 
\begin{eqnarray*} 
\mathcal{C}(\epsilon, I_W  ) \leq \mathcal{C} (\epsilon/2, I_W P_{\mathcal{U}}^N )\mathcal{C} (\epsilon/2, I_W P_{\mathcal{V}}^N ) \leq  \left[ 1+ \frac{4}{\epsilon} \left( \frac{1}{1-a} \right)^{1/2} \right]^ {2N}.
\end{eqnarray*}

Finally, by \eqref{Neps} we obtain the following estimate
\begin{align*}
    \ln(\mathcal{C} (\epsilon , I_W ))  & \leq 2N  \ln \left[1+ \frac{4}{\epsilon} \left( \frac{1}{1-a} \right)^{1/2} \right] \\
      %&= N \ln \left[  1+ \frac{2}{\epsilon (1-a)^{1/2}}\right] \\
      %& \leq  \frac{\ln \left( \mu_N ^2 {(1-a)}\right)}{\ln a} \ln \left( 4\left( \frac{1}{1-a} \right)^{1/2} \frac{1}{\epsilon} \right)\\
      %& = \frac{\ln\left(\frac{1}{(1-a)\mu_N^2}\right)} {\ln \left(\frac{1}{a}\right)}  \ln \left( 4\left( \frac{1}{1-a} \right)^{1/2} \frac{1}{\epsilon} \right)\\
      %& = \frac{2}{\ln{\frac{1}{a}}} \ln \left( \frac{1}{\mu_N (1-a)^{1/2}}\right) \ln \left( 4\left( \frac{1}{1-a} \right)^{1/2} \frac{1}{\epsilon} \right)\\
     % & \approx 4\frac{\ln(1/\epsilon)}{\ln(1/a)}\ln \left[1+ \frac{2}{\epsilon} \left( \frac{1}{1-a} \right)^{1/2} \right]\\
       & \approx 4\frac{[\ln \left(1/\epsilon \right)]^2}{\ln(1/a)},
      %\frac{2M}{\ln{(1/a)}} \ln \left( \frac{1}{\epsilon (1-a)^{1/2}}\right)   \ln \left( \frac{4}{\epsilon ({1-a})^{1/2}} \right),
    \end{align*}
and the proof is completed.
\end{proof}

\subsection{Lower bound}

In order to  present the lower bound for
%the numbers
$\mathcal{C} (\epsilon, I_K )$ we need the following lemma.
%the covering numbers  of operators. 

 \begin{lem}\label{CN1}\cite[Lemma 1]{Kuhn}
Let $X$ and $Y$ be d-dimensional real Hilbert spaces, and  $T: X  \rightarrow Y$ be an operator. For any  $\epsilon>0$, it holds
\begin{equation}
\mathcal{C}(\epsilon,T) \geq |\det \sqrt{T^{*}T}| \left(\frac{1}{\epsilon}\right) ^d.
\end{equation}
\end{lem}

A lower bound for $\mathcal{C} (\epsilon, I_W)$ is presented in the following theorem.

\begin{thm}\label{LB}
For any $\epsilon > 0$, the following inequality holds true
$$2\frac{[\ln \left(1/\epsilon \right)]^2}{\ln(1/a)} \leq \ln (\mathcal{C} (\epsilon, I_W)).
$$
\end{thm}

\begin{proof}

Let $E:= \{  \psi _{2k},\psi _{2k+1}:\, k\in \mathbb{N}\cup \{0\}\}$ be the orthonormal basis of $\mathcal{H}_W$ given in (\ref{baseRKHS}) and denoted by $$\psi_{2k}(x)=a^{k/2}\cos(b^k \pi x)\quad\text{and}\quad \psi_{2k+1}(x)=a^{k/2}\sin(b^k \pi x), \quad x\in I.$$ 

For $n\in\mathbb{N}$, consider the operator defined by the following composition
$$
T_n : E_n \stackrel{ J_n }{\longrightarrow} \mathcal{H}_W  \stackrel{I_W}{\longrightarrow}C(I)\stackrel{L_n}{\longrightarrow} L^2 ( I) \stackrel{P_n }{\longrightarrow} F_n,
$$
where $J_n$ is the embedding of $E_n :=$ span $\{\psi_{k}, : k=0,..,2n-1 \}$
%\hookrightarrow 
 into $ \mathcal{H}_W $, $L_n$ is the formal identity operator and
$P_n $ is the orthogonal projection of $L^2(I)$ onto $F_n := L_n I_W J_n(E_n)$.\ The operator $T_n :E_n \rightarrow F_n $ is uniquely determined by the following relation 
$$
T_n \psi_k (x) = \psi_k (x),\quad  k=0,1,\ldots,2n-1.
$$  
Hence
%If $A_n$
%$=(a_{i,j})_{i,j}^{n-1}$ 
%is the representing matrix of the operator $T_n ^* T_n: E_n \rightarrow E_n $  with respect to the   $\{ \psi_k, \,\,k=0,...,2n-1 \}$ of $E_n $, then for $0\leq  i,j\leq 2n-1$
$$\langle{T_n ^* T_n \psi_{j}},{\psi_{\ell}}\rangle_{W}=\langle{ T_n \psi_{j}},{T_n \psi_{\ell}}\rangle_{2}=\langle{  \psi_{j}},{\psi_{\ell}}\rangle_{2},\quad  0\leq  j,\ell\leq 2n-1.$$
Since  
$$\|\psi_{2k}\|_2^2=\|\psi_{2k+1}\|_2^2 =a^k,\quad k\geq 0,$$
it follows that
$$\langle{T_n ^* T_n \psi_{j}},{\psi_{\ell}}\rangle_{W}=
%=\langle{  \psi_{j}},{\psi_{\ell}}\rangle_{2}=\int_{-1}^{1}\psi_{j}(x)\psi_{\ell}(x)\,dx=
\delta_{j\ell}
\|\psi_{j}\|_2^2 =a^k,$$
for $j=2k$ and $j=2k+1$. That is, the representing matrix  of $T_n ^* T_n$ is the following diagonal matrix
$$\begin{bmatrix}
1 & 0 & 0 & 0 &\cdots & 0 & 0\\
0 & 1 & 0 & 0 &\cdots& 0 & 0\\
0 & 0 & a & 0 &\cdots& 0 & 0\\
0 & 0 & 0 & a &\cdots& 0 & 0\\
& & & & \cdots & & & \\
0 & 0 & 0 & 0 &\cdots&a^{n-1}& 0\\
0 & 0 & 0 & 0 &\cdots& 0 & a^{n-1}
\end{bmatrix},$$
and hence
$$\det(T_n ^* T_n)=a^{n(n-1)}.$$
From the norm estimates $\|L_n\|=\sqrt{2}$ and $\|P_n\|=\|J_n\|=1$,
by  an application of Lemma \ref{LemCN}-items (iii) and (v),  we have that
\begin{eqnarray*}
\mathcal{C} ( \sqrt{2}\epsilon, T_n )=\mathcal{C} (\sqrt{2}\epsilon, P_n L_n  I_W J_n ) & \leq & \mathcal{C} (1, P_n )\mathcal{C} (\sqrt{2}, L_n )\mathcal{C} ({\epsilon}, I_W )\mathcal{C} (1, J_ n )\\ & = & \mathcal{C} (\epsilon, I_W ).
\end{eqnarray*}
Now, Lemma \ref{CN1} implies that
$$
\mathcal{C} (\epsilon, I_W ) \geq \mathcal{C} (\sqrt{2}\epsilon, T_n  )\geq \sqrt{ \det (T_n ^* T_n )}\left(\frac{1}{\sqrt{2}\epsilon}\right)^{2n},
$$
then 
$$\mathcal{C} (\epsilon, I_W ) \geq \frac{a^{n^2 /2}}{(2 \epsilon^2 )^{n}}. 
$$

For $\epsilon>0$ sufficiently small, we can choose $n=\lceil \beta \rceil$, where
$$
\beta:= 2\frac{\ln(\sqrt{2}\epsilon)}{\ln(a)},
$$
and $\lceil \, \cdot\, \rceil$ is the ceiling function. And we obtain that 
 \begin{align*}
 \ln (\mathcal{C}(\epsilon, I_W ) &\geq 
  \frac{n^2}{2}\ln(a)-2n\ln(\sqrt{2}\epsilon)
  \approx 2\frac{[\ln \left(1/\epsilon \right)]^2}{\ln(1/a)}. 
  \end{align*} 
\end{proof}

\subsection{Proof of Theorem \ref{coveringesatimates} and final remarks}
A direct application of Theorem \ref{tma4.4} and Theorem \ref{LB} leads us to

%$$2\frac{\ln(1/\epsilon)^2}{\ln(1/a)} \leq \ln (\mathcal{C} (\epsilon, I_W)) \leq  4\frac{\ln(1/\epsilon)^2}{\ln(1/a)}.
%$$

%Therefore,  
\[ \frac{2}{{\ln\left(1/a\right)}}  \leq \liminf_{\epsilon\to 0}\frac{\ln(\mathcal{C}(\epsilon, I_W ))}{\phi(\epsilon)} \leq \limsup_{\epsilon\to 0}\frac{\ln(\mathcal{C}(\epsilon, I_W ))}{\phi(\epsilon)} \leq \frac{4}{{\ln\left(1/a\right)}}  ,\]
 where
 \[\phi(\epsilon):= [\ln \left(1/\epsilon \right)]^2 .\]

Unfortunately, we were not able to respond to the natural question 
of whether the limit $$\lim_{\epsilon\to 0}\frac{\ln(\mathcal{C}(\epsilon, I_W ))}{\phi(\epsilon)}$$
exists. As pointed out by Lorentz in \cite{Lorentz2}, the exact determination of the entropy is, in most cases, very difficult. See for instance \cite{Kuhn}, where it is left as an open question.

\section*{Acknowledgments} The authors thank the anonymous referee for carefully reading our manuscript and for several detailed corrections, suggestions, and comments offered. Especially, the authors are grateful for the kind contribution of the referee leading us to a corrected version of the main result and, consequently, to an enhanced and publishable version of the paper.

\end{document}